\documentclass[10pt]{article}

\usepackage{amsmath,amsthm,amssymb,bbm,multicol}
\usepackage[bookmarksnumbered=true,pagebackref]{hyperref}
\usepackage{tocbibind}
\usepackage{enumitem}
\usepackage{graphicx,float,caption}
\usepackage{pgfplots}
\pgfplotsset{compat = newest}
\usetikzlibrary{arrows.meta}

\newtheorem{theorem}{Theorem}
\newtheorem{proposition}{Proposition}
\newtheorem{lemma}{Lemma}
\newtheorem{corollary}{Corollary}

\theoremstyle{remark}
\newtheorem{remark}{Remark}

\theoremstyle{definition}

\newtheoremstyle{notes}
{3pt}
{3pt}
{}
{}
{\bfseries}
{:}
{.4em}
{}
\theoremstyle{notes}
\newtheorem*{keywords}{Keywords}
\newtheorem*{subjclass}{AMS MSC 2020}

\DeclareMathOperator{\spn}{span}
\DeclareMathOperator{\proj}{proj}
\DeclareMathOperator{\conv}{conv}

\usepackage[draft]{todonotes} 
\title{Bounds on some geometric functionals of high dimensional Brownian convex hulls and their inverse processes}

\author{
Hugo Panzo 
\\
\href{mailto:hugo.panzo@slu.edu}{\texttt{{\small hugo.panzo@slu.edu}}}
\and
Evan Socher
\\ 
\href{mailto:evan.socher@slu.edu}{\texttt{{\small evan.socher@slu.edu}}}
}

\date{\today}

\begin{document}

\maketitle

\begin{abstract}
We prove two-sided bounds on the expected values of several geometric functionals of the convex hull of Brownian motion in $\mathbb{R}^n$ and their inverse processes. This extends some recent results of McRedmond and Xu \cite{McRedmond_Xu}, Jovaleki\'{c} \cite{Jovalekic}, and Cygan, \v{S}ebek, and the first author \cite{hull_bounds} from the plane to higher dimensions. Our main result shows that the average time required for the convex hull in $\mathbb{R}^n$ to attain unit volume is at most $n\sqrt[n]{n!}$. The proof relies on a novel procedure that embeds an $n$-simplex of prescribed volume within the convex hull of the Brownian path run up to a certain stopping time. All of our bounds capture the correct order of asymptotic growth or decay in the dimension $n$.
\end{abstract}

\begin{keywords}
Brownian motion, convex hull, convex optimization, inverse process, diameter, circumradius. 
\end{keywords}

\begin{subjclass}
Primary 60D05, 60J65; Secondary 52A20, 90C25.
\end{subjclass}


\section{Introduction}

Let $\boldsymbol{W}=\{\boldsymbol{W}_t:t\geq 0 \}$ denote Brownian motion in $\mathbb{R}^n$ starting at the origin $\boldsymbol{0}$. That is,  
\[
\boldsymbol{W}_t=\left(W_t^{(1)},\dots, W_t^{(n)}\right), t\geq 0,
\]
where each coordinate $\{W_t^{(j)}:t\geq 0\}$, $j=1,\dots, n$, is standard Brownian motion in $\mathbb{R}$, independent of the other coordinates. For a set $\mathcal{A}\subset\mathbb{R}^n$, let $\conv \mathcal{A}$ denote the \emph{convex hull} of $\mathcal{A}$. In other words, $\conv \mathcal{A}$ is the smallest convex subset of $\mathbb{R}^n$ that contains $\mathcal{A}$. Define $\mathcal{H}_t$ to be the convex hull of the path of $\boldsymbol{W}$ run up to time $t$, namely,
\[
\mathcal{H}_t:=\conv\{\boldsymbol{W}_s:0\leq s\leq t \}.
\]

This paper is concerned with estimating the expected values of several geometric functionals of $\mathcal{H}_t$ and their inverse processes. The particular functionals of $\mathcal{H}_t$ that we study are volume, surface area, diameter, and circumradius. For a \emph{convex body} $\mathcal{K}\subset\mathbb{R}^n$ (nonempty, convex, and compact subset), let $V(\mathcal{K})$ denote its $n$-dimensional Lebesgue measure, $S(\mathcal{K})$ the $n-1$ dimensional Lebesgue measure of its topological boundary, $D(\mathcal{K})$ its diameter, and $R(\mathcal{K})$ its circumradius. Since the path $\{\boldsymbol{W}_s:0\leq s\leq t \}\subset\mathbb{R}^n$ is almost surely compact, it follows that $\mathcal{H}_t$ is almost surely a convex body. Hence, we can define the processes $V_t$, $S_t$, $D_t$, and $R_t$ by
\begin{equation}\label{eq:process}
X_t:=X(\mathcal{H}_t),~t\geq 0,
\end{equation}
where $X\in\{V,S,D,R\}$.

As these four processes are almost surely nondecreasing functions of time, we can also study their right-continuous \emph{inverse processes}. The inverse process tells us how long we must wait for the functional to exceed a given value. More precisely, we have the definition 
\begin{equation}\label{eq:inverse}
\Theta_y^X:=\inf\{t\geq 0:X_t>y\},~y\geq 0,
\end{equation}
where $X\in\{V,S,D,R\}$. As remarked upon in \cite{hull_bounds}, the functionals that we consider can be used to quantify the size of $\mathcal{H}_t$, while their inverse processes provide some information on its speed of growth. The scaling properties of Brownian motion, Lebesgue measure, and Euclidean distance imply that we can limit our study to the expected values of these functionals and their inverse processes at time $t=1$ without any loss of generality; see Proposition \ref{prop:distributional_equalities} below. 

The study of the convex hull of Brownian motion has a long history going back to L\'{e}vy in the 1940's; see \cite{Levy}. Let us summarize some more recent results that involve the functionals we are interested in. Most impressive are the explicit formulas for the expected values of $V_1$ and $S_1$ that hold in all dimensions. These expressions were derived by Eldan in \cite{Eldan} and are given by
\begin{equation}\label{eq:Eldan}
\mathbb{E}\left[V_1\right]=\left(\frac{\pi}{2}\right)^{n/2}\frac{1}{\Gamma(1+n/2)^2},~~~\mathbb{E}\left[S_1\right]=\frac{2(2\pi)^{(n-1)/2}}{\Gamma(n)}.
\end{equation}
The formula for $\mathbb{E}[V_1]$ with $n=2$ had appeared previously in \cite{El_Bachir}, while for $\mathbb{E}[S_1]$, the formula had been derived for $n=2$ in \cite{Takacs_solution} and for $n=3$ in \cite{stable_hull}. The formulas \eqref{eq:Eldan} for all dimensions were subsequently recovered in \cite{Sobolev_balls} using different methods.

In contrast to the work of Eldan, the papers \cite{McRedmond_Xu}, \cite{Jovalekic}, and \cite{hull_bounds} deal exclusively with the planar case $n=2$ and are limited to estimates instead of exact formulas like \eqref{eq:Eldan}. However, these papers treat different geometric functionals of $\mathcal{H}_1$ which are seemingly not amenable to the methods of \cite{Eldan} and \cite{Sobolev_balls}. More specifically, \cite{McRedmond_Xu} and \cite{Jovalekic} derive bounds for the expected diameter of $\mathcal{H}_1$ when $n=2$, while \cite{hull_bounds} does the same for the expected circumradius and inradius. Moreover, \cite{hull_bounds} initiates the study of the inverse processes \eqref{eq:inverse} of all five geometric functionals (volume, surface area, diameter, circumradius, and inradius) by computing two-sided bounds on their expected values when $n=2$. The paper \cite{hull_bounds} also complements its bounds with estimates from extensive Monte Carlo simulations. 

The main contributions of the present paper are to extend most of the bounds derived in \cite{McRedmond_Xu}, \cite{Jovalekic}, and \cite{hull_bounds} from the plane to higher dimensions. All of our bounds capture the correct order of asymptotic growth or decay in the dimension $n$ in the sense that the upper and lower bounds are asymptotically equivalent up to a constant factor as $n\to\infty$. We were unable to obtain bounds with matching orders of asymptotic growth or decay for the expected values of the inradius and its inverse process so we leave those cases for future investigation.


\section{Main results}

Our first two theorems concern the expected values of the inverse processes of the volume and surface area functionals that were defined in \eqref{eq:process} and \eqref{eq:inverse}. These extend the corresponding bounds of \cite{hull_bounds} from the plane to higher dimensions and also complement Eldan's exact formulas \eqref{eq:Eldan} for the expected values of the functionals themselves. These theorems are proved in Section \ref{sec:VS_inverse}. We remark that the proof of the upper bound in Theorem \ref{thm:inverse_volume} required significantly more work than any of the other theorems in this paper. 

\begin{theorem}\label{thm:inverse_volume}
For any dimension $n\geq 1$, the inverse volume process satisfies 
\[
\frac{2}{\pi}\Gamma(1+n/2)^{4/n}\leq\mathbb{E}\left[\Theta_1^V\right]\leq n\sqrt[n]{n!}.
\]
\end{theorem}

\begin{remark}\label{rem:high_dimensions}
When $n=2$, the upper bound from \cite{hull_bounds} is better than that of Theorem \ref{thm:inverse_volume}, since the former uses the minimum of two independent inverse range processes in the first stage instead of the exit time of a two dimensional ball. The latter method is more favorable in higher dimensions, however, not only for its computational simplicity, but because for fixed range and radius, the expected value of the minimum of $n$ independent inverse range processes decays like $1/\log n$, while that of the exit time of an $n$-dimensional ball decays like $1/n$; see also Remarks \ref{rem:diameter} and \ref{rem:circumradius}.
\end{remark}

\begin{theorem}\label{thm:inverse_area}
For any dimension $n\geq 2$, the inverse surface area process satisfies 
\[
\frac{1}{2\pi}\left(\frac{1}{2}\Gamma(n)\right)^{2/(n-1)}\leq \mathbb{E}\left[\Theta_1^S\right]\leq \left(\frac{\sqrt{\pi}\, n}{\sqrt[n]{\Gamma(1+n/2)}}\right)^{2/(n-1)} n\sqrt[n]{n!}.
\]
\end{theorem}

The asymptotic behavior of these bounds is straightforward to deduce from Stirling's approximation and can be summarized in the following corollary. This result verifies our claim that the upper and lower bounds for the expected values of both inverse processes have the same asymptotic order as $n\to\infty$, namely, $n^2$.

\begin{corollary}
For either $X=V$ or $X=S$, we have
\[
\liminf_{n\to\infty}\frac{\mathbb{E}\left[\Theta_1^X\right]}{n^2}\geq \frac{1}{2\pi e^2}~\text{ and }~\limsup_{n\to\infty}\frac{\mathbb{E}\left[\Theta_1^X\right]}{n^2} \leq \frac{1}{e}.
\]
\end{corollary}

Our next two theorems involve the diameter functional and its inverse process. For these results, it is plain to see that the upper and lower bounds have matching asymptotic orders as $n\to\infty$. The proofs are given in Section \ref{sec:diameter}.

\begin{theorem}\label{thm:diameter}
For any dimension $n\geq 1$, the diameter process satisfies 
\[
\sqrt{n}\leq\mathbb{E}\left[D_1\right]\leq 2\sqrt{\log 2}\sqrt{n}.
\]
\end{theorem}

\begin{theorem}\label{thm:inverse_diameter}
For any dimension $n\geq 1$, the inverse diameter process satisfies 
\[
\frac{1}{4\log 2}\frac{1}{n}\leq\mathbb{E}\left[\Theta_1^D\right]\leq \frac{1}{n}.
\]
\end{theorem}

\begin{remark}\label{rem:diameter}
When $n=2$, the lower bounds from \cite{McRedmond_Xu,Jovalekic} and the upper bound from \cite{hull_bounds} are better than those of Theorems \ref{thm:diameter} and \ref{thm:inverse_diameter}, respectively, for the same reason described in Remark \ref{rem:high_dimensions}. However, just like in that case, the methods of the present paper become more effective in high dimensions.
\end{remark}

Our last two theorems deal with the circumradius functional and its inverse process. For these results, it is also clear that the upper and lower bounds have the same asymptotic order as $n\to\infty$. The proofs are given in Section \ref{sec:circumradius}.

\begin{theorem}\label{thm:circumradius}
For any dimension $n\geq 1$, the circumradius process satisfies 
\[
\frac{1}{2}\sqrt{n}\leq\mathbb{E}\left[R_1\right]\leq \sqrt{\log2}\sqrt{n}.
\]
\end{theorem}

\begin{theorem}\label{thm:inverse_circumradius}
For any dimension $n\geq 1$, the inverse circumradius process satisfies 
\[
\frac{1}{\log 2}\frac{1}{n}\leq\mathbb{E}\left[\Theta^R_1\right]\leq 4\frac{1}{n}.
\]
\end{theorem}

\begin{remark}\label{rem:circumradius}
When $n=2$, the lower bound of Theorem \ref{thm:circumradius} and the upper bound of Theorem  \ref{thm:inverse_circumradius} are worse than those of \cite{hull_bounds}. This is also due to the reason described in Remark \ref{rem:high_dimensions}. Similarly to that case, the methods of the present paper become more effective in high dimensions.
\end{remark}

\begin{remark}
When $n=1$, $V_t$, $D_t$, and $2R_t$ are nothing but the \emph{range} of $\boldsymbol{W}$. Hence, their inverse processes can be expressed in terms of the \emph{inverse range}. The distributions of the range and its inverse are known explicitly; see \cite{Feller,Imhof}.
\end{remark}


\section{Preliminaries}

Following \cite{Schneider}, we denote the volume and surface area of the unit ball in $\mathbb{R}^n$ by $\kappa_n$ and $\omega_n$, respectively. These quantities are given by the well-known formulas 
\begin{equation}\label{eq:unit_ball}
\kappa_n=\frac{\pi^{n/2}}{\Gamma(1+n/2)},~~~\omega_n=\frac{2\pi^{n/2}}{\Gamma(n/2)}.
\end{equation}

Another well-known formula that is essential to our results is that of the mean exit time of $n$-dimensional Brownian motion from a ball of radius $r>0$ in $\mathbb{R}^n$ when it starts from the center of the ball. Let $\tau_r$ denote this exit time. Then we have
\begin{equation}\label{eq:ball_moment}
\mathbb{E}[\tau_r]=\frac{r^2}{n}.
\end{equation}
This formula can be deduced from a routine martingale or PDE argument; see Problem 4.2.25 of \cite{K&S}.

As mentioned in the introduction, the familiar scaling properties of Brownian motion, Lebesgue measure, and Euclidean distance lead to convenient distributional identities that allow us to focus our study on the expected values of the geometric functionals of $\mathcal{H}_t$ and their inverse processes at time $t=1$ without any loss of generality. These distributional identities are the $n$-dimensional version of \cite[Proposition 4.1]{hull_bounds} and are listed in the following proposition. We stress that these distributional identities hold for fixed times and not as processes. 

\begin{proposition}\label{prop:distributional_equalities}\
Let $n\geq 2$ and consider the $V$, $S$, $D$, and $R$ functionals of the convex hull of $n$-dimensional Brownian motion along with their inverse processes that were defined in \eqref{eq:inverse}. Then for all $t\geq 0$ and $y\geq 0$, we have the following equalities in distribution:
\begin{enumerate}[label=(\roman*)]

\item $\displaystyle V_t\stackrel{d}{=}t^{n/2}V_1,~~\Theta_y^V\stackrel{d}{=}y^{2/n}\Theta_1^V,~~\Theta_1^V\stackrel{d}{=}V_1^{-2/n}$;

\item $\displaystyle S_t\stackrel{d}{=}t^{(n-1)/2}S_1,~~\Theta_y^S\stackrel{d}{=}y^{2/(n-1)}\Theta_1^S,~~\Theta_1^S\stackrel{d}{=}S_1^{-2/(n-1)}$;

\item $\displaystyle D_t\stackrel{d}{=}\sqrt{t} D_1,~~\Theta_y^D\stackrel{d}{=}y^2\Theta_1^D,~~\Theta_1^D\stackrel{d}{=}D_1^{-2}$;

\item $\displaystyle R_t\stackrel{d}{=}\sqrt{t} R_1,~~\Theta_y^R\stackrel{d}{=}y^2\Theta_1^R,~~\Theta_1^R\stackrel{d}{=}R_1^{-2}$.

\end{enumerate}

\end{proposition}

\begin{proof}[Proof of Proposition \ref{prop:distributional_equalities}]
We follow the proof of Proposition 4.1 in \cite{hull_bounds}. The scaling property of Brownian motion implies that for any $t\geq 0$ and $\lambda>0$, we have
\begin{align}
\mathcal{H}_{\lambda t}=\conv\{\boldsymbol{W}_s:0\leq s \leq \lambda t\}&=\conv\{\boldsymbol{W}_{\lambda s}:0\leq s \leq t\}\nonumber \\
&\stackrel{d}{=}\sqrt{\lambda}\conv\{\boldsymbol{W}_{s}:0\leq s \leq t\}\nonumber \\
&=\sqrt{\lambda}\mathcal{H}_t.\label{eq:hull_scaling}
\end{align}
The distributional identities of Proposition \ref{prop:distributional_equalities} all follow from \eqref{eq:hull_scaling}. As the proofs are similar, we give the details only for part (i) and leave the rest to the reader.  

The first distributional identity of part (i) is trivial when $t=0$, so there is no loss of generality in assuming that $t>0$. For $u\geq 0$, we can use \eqref{eq:hull_scaling} to write
\begin{equation}\label{eq:volume_scaling}
V_{t u}=V(\mathcal{H}_{t u})\stackrel{d}{=}V(\sqrt{t}\mathcal{H}_u) =t^{n/2} V_u.
\end{equation}
Taking $u=1$ proves the identity. 

Similarly to the first identity, we can also assume that $y>0$ when proving the second identity of part (i). Now for any $t\geq 0$, we can use \eqref{eq:volume_scaling} to write
\begin{align*}
\mathbb{P}\left(\Theta_y^V>t\right)=\mathbb{P}(V_t<y)=\mathbb{P}\left(\frac{1}{y}V_t<1\right)&=\mathbb{P}\left(V_{t/y^\frac{2}{n}}<1\right)\\
&=\mathbb{P}\left(\Theta_1^V>t/y^\frac{2}{n}\right)\\
&=\mathbb{P}\left(y^\frac{2}{n}\Theta_1^V>t\right).
\end{align*}
This shows that $\Theta_y^V$ and $y^{2/n}\Theta_1$ have the same distribution.

The proof of the last identity of part (i) can be deduced similarly via 
\[
\mathbb{P}\left(\Theta_1^V>t\right)=\mathbb{P}(V_t<1)=\mathbb{P}\left(t^{n/2}V_1<1\right)=\mathbb{P}\left(t<V_1^{-2/n}\right).
\]
\end{proof}


\section{Proofs of the main results}

\subsection{Inverse processes of volume and surface area}\label{sec:VS_inverse}

The proof of the upper bound in Theorem \ref{thm:inverse_area} employs an $n$-stage construction which stops and restarts $\boldsymbol{W}$ as it exits a sequence of hypercylinders of decreasing spherical dimension. This procedure allows us to embed a $n$-simplex $\mathcal{S}$ with prescribed volume within the convex hull of $\boldsymbol{W}$ at a certain stopping time $T_n$. This is essentially an $n$-dimensional version of the idea used to prove Proposition 1.6 of \cite{hull_bounds}. The construction is parameterized by a sequence of $n$ positive numbers $r_1,r_2,\dots,r_n$. These parameters are the radii of the spherical parts of the aforementioned hypercylinders. Upon completion of the procedure at time $T_n$, we necessarily have $V(\mathcal{H}_{T_n})\geq V(\mathcal{S})$, whence we deduce the upper bound
\begin{equation}\label{eq:simplex_bound}
\mathbb{E}\left[\Theta_{V(\mathcal{S})}^V\right] \leq \mathbb{E}[T_n].
\end{equation}
After computing $V(\mathcal{S})$ and $\mathbb{E}[T_n]$ explicitly in terms of the parameters, we use the forthcoming Lemma \ref{lem:optimization} to optimize the right-hand side of \eqref{eq:simplex_bound} under the constraint $V(\mathcal{S})\geq 1$ in order to obtain the best possible bound using this method.

As alluded to above, the upper bound of Theorem \ref{thm:inverse_volume} requires solving a convex optimization problem whose solution we split into the following two lemmas. The first lemma verifies that certain functions are indeed convex and the second lemma solves the optimization problem. We refer to \cite{Rockafellar} for the requisite convex analysis theory.

\begin{lemma}\label{lem:convexity}
Let $n\in\mathbb{N}$, and define the functions $f_n$ and $g_n$ by
\[
f_n(x_1,\dots,x_n):=\sum_{j=1}^n \frac{x_j^2}{n-j+1},~~(x_1,\dots,x_n)\in\mathbb{R}_{>0}^n,
\]
\[
g_n(x_1,\dots,x_n):=\frac{1}{x_1\cdots x_n},~~(x_1,\dots,x_n)\in\mathbb{R}_{>0}^n.
\]
Then $f_n$ and $g_n$ are both convex functions on the positive orthant $\mathbb{R}_{>0}^n$.
\end{lemma}

\begin{proof}[Proof of Lemma \ref{lem:convexity}]
We establish convexity by showing that the Hessian matrices of $f_n$ and $g_n$ are both positive semi-definite on the positive orthant $\mathbb{R}_{>0}^n$; see \cite[Theorem 4.5]{Rockafellar}. This is a trivial matter for $f_n$, since its Hessian matrix is a constant diagonal matrix with positive diagonal entries. The case of $g_n$ requires a bit more work but is also straightforward. Indeed, routine calculations show that the entries of the Hessian matrix $\boldsymbol{H}$ of $g_n$ are given by
\[
h_{jk}=\begin{cases}
\displaystyle\frac{1}{x_1\cdots x_n}\,\frac{2}{x_j^2},&\text{ if }j=k,\\
\displaystyle\frac{1}{x_1\cdots x_n}\,\frac{1}{x_j x_k},&\text{ if }j\neq k.
\end{cases}
\]
Now if $\boldsymbol{z}=(z_1,\dots, z_n)$ is any vector in $\mathbb{R}^n$ and $(x_1,\dots,x_n)\in\mathbb{R}_{>0}^n$, we can write
\begin{align*}
\langle\boldsymbol{z},\boldsymbol{H}\boldsymbol{z}\rangle&=\sum_{j=1}^n z_j\sum_{k=1}^n h_{jk}\, z_k\\
&=\frac{1}{x_1\cdots x_n}\sum_{j=1}^n \left(\frac{2z_j^2}{x_j^2}+\sum_{k\neq j} \frac{z_j z_k}{x_j x_k}\right)\\
&=\frac{1}{x_1\cdots x_n}\left(\left(\sum_{j=1}^n \frac{z_j^2}{x_j^2}\right)+\left(\sum_{j=1}^n \frac{z_j}{x_j} \right)^2 \right)\\
&\geq 0.
\end{align*}
\end{proof}

\begin{lemma}\label{lem:optimization}
For any $n\in \mathbb{N}$ we have
\begin{equation}\label{eq:convex_program}
\inf\left\{\sum_{j=1}^n \frac{x_j^2}{n-j+1}\middle|(x_1,\dots,x_n)\in \mathbb{R}_{> 0}^n\text{ \emph{and} }x_1\cdots x_n\geq n!\right\}=n\sqrt[n]{n!}.
\end{equation}
\end{lemma}

\begin{proof}[Proof of Lemma \ref{lem:optimization}]
Following the notation and results of Lemma \ref{lem:convexity}, in order to prove Lemma \ref{lem:optimization}, we need to minimize the convex objective function $f_n$ over the convex domain $\mathbb{R}_{>0}^n$, subject to the convex constraint $n!\,g_n-1\leq 0$. In the nomenclature of \cite{Rockafellar}, this is an \emph{ordinary convex program} and we appeal to Theorem 28.3 of that reference for a solution. Towards this end, we claim that the infimum on the left-hand side of \eqref{eq:convex_program} is attained at 
\[
\boldsymbol{\bar{x}}=\left(\sqrt{n}\sqrt[2n]{n!},\sqrt{n-1}\sqrt[2n]{n!},\dots,\sqrt[2n]{n!}\right),
\]
and that the Kuhn-Tucker coefficient corresponding to the constraint is $\lambda=2\sqrt[n]{n!}$. We verify this by checking the three conditions of \cite[Theorem 28.3]{Rockafellar}.\\

\noindent \textbf{Condition (a)}: inequality constraints

It is clear that $\lambda\geq 0$. Moreover, $n!\,g_n(\boldsymbol{\bar{x}})-1\leq 0$ and $\lambda(n!\,g_n(\boldsymbol{\bar{x}})-1)= 0$ both follow from 
\[
n!\,g_n(\boldsymbol{\bar{x}})-1=\frac{n!}{\sqrt{n!}\sqrt{n!}}-1=0.
\]

\noindent \textbf{Condition (b)}: equality constraints

This condition is vacuously satisfied since there are no equality constraints.\\

\noindent \textbf{Condition (c)}: Lagrangian 

Since $f_n$ and $g_n$ are both differentiable on $\mathbb{R}_{>0}^n$, \textbf{Condition (c)} becomes a statement about the gradient of the Lagrangian instead of its subdifferential. In particular, routine calculations show that
\begin{align}
\nabla f_n(\boldsymbol{x})\big|_{\boldsymbol{x}=\boldsymbol{\bar{x}}}&=\left(\frac{2x_1}{n},\frac{2x_2}{n-1},\dots, \frac{2x_n}{1} \right)\Big|_{\boldsymbol{x}=\boldsymbol{\bar{x}}}\nonumber \\
&=\left(\frac{2\sqrt[2n]{n!}}{\sqrt{n}},\frac{2\sqrt[2n]{n!}}{\sqrt{n-1}},\dots,2\sqrt[2n]{n!}\right)\label{eq:grad_objective}
\end{align}
and
\begin{align}
\lambda \nabla \big(n!\,g_n(\boldsymbol{x})-1\big)\big|_{\boldsymbol{x}=\boldsymbol{\bar{x}}}&=-2\sqrt[n]{n!}\frac{n!}{x_1\cdots x_n}\left(\frac{1}{x_1},\dots, \frac{1}{x_n}\right)\Big|_{\boldsymbol{x}=\boldsymbol{\bar{x}}}\nonumber \\
&=-2\sqrt[n]{n!}\left(\frac{1}{\sqrt{n}\sqrt[2n]{n!}},\frac{1}{\sqrt{n-1}\sqrt[2n]{n!}},\dots,\frac{1}{\sqrt[2n]{n!}}\right)\nonumber \\
&=-\left(\frac{2\sqrt[2n]{n!}}{\sqrt{n}},\frac{2\sqrt[2n]{n!}}{\sqrt{n-1}},\dots,2\sqrt[2n]{n!}\right).\label{eq:grad_constraint}
\end{align}
Adding \eqref{eq:grad_objective} and \eqref{eq:grad_constraint} demonstrates that  
\[
\nabla \Big(f_n(\boldsymbol{x})+\lambda \big(n!\,g_n(\boldsymbol{x})-1\big)\Big)\Big|_{\boldsymbol{x}=\boldsymbol{\bar{x}}}=\boldsymbol{0}.
\] 
This checks \textbf{Condition (c)} and verifies our claim.

Finally, we prove Lemma \ref{lem:optimization} by evaluating the objective function at $\boldsymbol{\bar{x}}$, namely,
\[
f_n(\boldsymbol{\bar{x}})=\sum_{j=1}^n \frac{(n-j+1)\sqrt[n]{n!}}{n-j+1}=n\sqrt[n]{n!}.
\]
\end{proof}

With Lemmas \ref{lem:convexity} and \ref{lem:optimization} in hand, we can begin to prove Theorem \ref{thm:inverse_volume}. The proof of Theorem \ref{thm:inverse_area} is much simpler and appears at the end of this section.

\begin{proof}[Proof of Theorem \ref{thm:inverse_volume}]
The lower bound is a straightforward consequence of part (i) of Proposition \ref{prop:distributional_equalities} together with Jensen's inequality and Eldan's formula \eqref{eq:Eldan} for the expected value of the volume. In particular, these results allow us to write
\begin{align*}
\mathbb{E}\left[\Theta_1^V\right]=\mathbb{E}\left[V_1^{-2/n}\right] &\geq \left(\left(\frac{\pi}{2}\right)^{n/2}\frac{1}{\Gamma(1+n/2)^2}\right)^{-2/n}\\
&=\frac{2}{\pi}\Gamma(1+n/2)^{4/n}.
\end{align*}

Proving the upper bound requires the $n$-stage procedure that was described at the beginning of Section \ref{sec:VS_inverse}. Starting with $T_0=0$, $\boldsymbol{x}_0=\boldsymbol{0}$, and $\mathcal{X}_0=\mathbb{R}^n$, we recursively define $T_j$, $\boldsymbol{x}_j$, and $\mathcal{X}_j$ for $j=1,\dots,n$ by 
\[
T_j=\inf\left\{t>T_{j-1}:\big\|\proj(\boldsymbol{W}_t,\mathcal{X}_{j-1})\big\|\geq r_j\right\},
\]
\[
\boldsymbol{x}_j=\boldsymbol{W}_{T_j},
\]
\[
\mathcal{X}_j=\spn\{\boldsymbol{x}_0,\dots,\boldsymbol{x}_j\}^\perp.
\]
In words, $T_j$ is the first time after $T_{j-1}$ that the orthogonal projection of $\boldsymbol{W}$ onto the linear subspace $\mathcal{X}_{j-1}$ leaves the centered open ball of radius $r_j$, the point $\boldsymbol{x}_j$ is the position of $\boldsymbol{W}$ at time $T_j$, and $\mathcal{X}_j$ is the orthogonal complement in $\mathbb{R}^n$ of the linear span of the vectors $\boldsymbol{x}_0,\dots,\boldsymbol{x}_j$.

With this construction, it is clear that $\{\boldsymbol{x}_0,\boldsymbol{x}_1,\dots,\boldsymbol{x}_n\}\subset\mathcal{H}_{T_n}$. Moreover, 
\begin{align*}
\|\proj\left(\boldsymbol{x}_j,\spn\{\boldsymbol{x}_0,\dots,\boldsymbol{x}_{j-1}\}^\perp\right)\|&=\|\proj\left(\boldsymbol{W}_{T_j},\mathcal{X}_{j-1}\right)\|\\
&= r_j>0,~1\leq j\leq n
\end{align*}
implies that $\{\boldsymbol{x}_1,\dots,\boldsymbol{x}_n\}$ is a linearly independent set of vectors. In particular, it follows that $\boldsymbol{x}_0,\boldsymbol{x}_1,\dots,\boldsymbol{x}_n$ are affinely independent, and therefore constitute the vertices of some $n$-simplex $\mathcal{S}\subset\mathcal{H}_{T_n}$. Since the $n$-parallelotope spanned by the vectors $\boldsymbol{x}_1,\dots,\boldsymbol{x}_n$ can be partitioned into $n!$ copies of $\mathcal{S}$, we deduce that 
\begin{equation}\label{eq:simplex_vol}
V(\mathcal{S})=\frac{1}{n!}\sqrt{\det(\boldsymbol{G})},
\end{equation}
where $\boldsymbol{G}$ is the Gram matrix of the vectors $\boldsymbol{x}_1,\dots,\boldsymbol{x}_n$. Letting $\boldsymbol{M}$ denote the $n\times n$ matrix with columns $\boldsymbol{x}_1,\dots,\boldsymbol{x}_n$, we can write
\begin{equation}\label{eq:Gram_det}
\det(\boldsymbol{G})=\det(\boldsymbol{M}^\top \boldsymbol{M})=\det(\boldsymbol{M})^2.
\end{equation}

The rotational invariance of $\boldsymbol{W}$ allows us to reorient the coordinate axes in a convenient way without affecting the distribution of the stopping times $T_1,\dots,T_n$. This will considerably simplify the computation of $\det(\boldsymbol{G})$. Equivalently, we can apply an orthogonal transformation to each $\boldsymbol{x}_j$ that changes the $j$th coordinate to $r_j$ and the last $n-j$ coordinates to $0$, while fixing the first $j-1$ coordinates. Hence, without loss of generality, we can take $\boldsymbol{x}_j=(x_{1j},\dots, x_{nj})$ with $x_{kj}=r_j$ if $k=j$ and $x_{kj}=0$ if $k>j$. See Figure \ref{fig:procedure} for an illustration of each stage of the construction of $\mathcal{S}$ and reorientation when $n=3$. In particular, this makes $\boldsymbol{M}$ an upper triangular matrix with diagonal entries $r_1,\dots,r_n$. Therefore, $\det(\boldsymbol{M})=r_1\cdots r_n$, and we can conclude from \eqref{eq:simplex_vol} and \eqref{eq:Gram_det} that
\begin{equation}\label{eq:simplex_volume}
V(\mathcal{H}_{T_n})\geq V(\mathcal{S})=\frac{r_1\cdots r_n}{n!}.
\end{equation}


\begin{figure}

\noindent\resizebox{\textwidth}{!}{

\begin{tikzpicture}
    \begin{axis}[
            xmin=-2,xmax=2,
            ymin=-2,ymax=1.5,
            ticks=none,width = 3.in, height = 2.7in]]
            \addplot [domain=0:360,samples=50]({cos(x)},{sin(x)}); 
            \addplot[domain=0:180,samples=50,dashed]({cos(x)},{.3*sin(x)});
            \addplot[domain=180:360,samples=50]({cos(x)},{.3*sin(x)});

            \addplot[domain=90:270,samples=50,dashed]({.5*cos(x)},{sin(x)});
            \addplot[domain=-90:90,samples=50]({.5*cos(x)},{sin(x)});
        \addplot[domain=0:.35355,samples=50,red,dashed]({x},{.70710/.35355*x});
            \fill (0,0) circle[radius=2pt];
            \fill (.35355,.70710) circle[radius=2pt];
            \node[left=.5pt of {(.35355,.70710)},outer sep=1pt]{$\boldsymbol{x}'_1$};
            \node[below =.5pt of {(0,0)},outer sep=1pt]{$\boldsymbol{x}_0$};
            \addplot[domain=-1:0),samples=50]({x},{0});
            \node[above=.01pt of {(-.5,0)},outer sep=.2pt]{$r_1$};
            \node[above =.5pt of {(-1.3,-1.8)},outer sep=1pt]{$\boldsymbol{x}'_1=\begin{bmatrix}x_{11}\\x_{21}\\x_{31}\end{bmatrix}$};
            \draw[{Latex[length=1.5mm]}-{Latex[length=1.5mm]}] (-1,0)--(0,0);
    \end{axis}
\end{tikzpicture}


\begin{tikzpicture}
    \begin{axis}[
            xmin=-2,xmax=2,
            ymin=-2,ymax=1.5,
            ticks=none,width = 3in, height = 2.7in]]
            \addplot [domain=0:360,samples=50]({cos(x)},{sin(x)}); 
            \addplot[domain=0:180,samples=50,dashed]({cos(x)},{.3*sin(x)});
            \addplot[domain=180:360,samples=50]({cos(x)},{.3*sin(x)});
            \node[right=.5pt of {(1,.1)},outer sep=1pt]{$\boldsymbol{x}_1$};
            \node[above =.5pt of {(0,0)},outer sep=1pt]{$\boldsymbol{x}_0$};
            \addplot[domain=0:1),samples=50,red]({x},{0});
              \addplot[domain=-1:0),samples=50]({x},{0});
            \node[above=.01pt of {(-.5,0)},outer sep=.2pt]{$r_1$};
            \node[above =.5pt of {(-1.3,-1.8)},outer sep=1pt]{$\boldsymbol{x}_1=\begin{bmatrix}r_1\\0\\0\end{bmatrix}$};
            \fill (0,0) circle[radius=2pt];
            \fill (1,0) circle[radius=2pt];
            \draw[{Latex[length=1.5mm]}-{Latex[length=1.5mm]}] (-1,0)--(0,0);
    \end{axis}
\end{tikzpicture}

}

\vspace{.8 mm}


\noindent\resizebox{\textwidth}{!}{

\begin{tikzpicture}
    \begin{axis}[
            xmin=-2,xmax=2,
            ymin=-2,ymax=1.2,
            ticks=none,width = 3in, height = 2.5in]
            \addplot[domain=90:3*90,samples=50]({.35*.4*cos(x)+sqrt(2)},{(250/189)*.4*sin(x)});
            \addplot[domain=3*90:5*90,samples=50,dashed]({.35*.4*cos(x)+sqrt(2)},{(250/189)*.4*sin(x)});
            \addplot[domain=90:3*90,samples=50]({.35*.4*cos(x)+sqrt(2)-.8},{(250/189)*.4*sin(x)});
            \addplot[domain=3*90:5*90,samples=50,dashed]({.35*.4*cos(x)+sqrt(2)-.8},{(250/189)*.4*sin(x)});
            \addplot[domain=90:3*90,samples=50]({.35*.4*cos(x)-sqrt(2)},{(250/189)*.4*sin(x)});
            \addplot[domain=3*90:5*90,samples=50,dashed]({.35*.4*cos(x)-sqrt(2)},{(250/189)*.4*sin(x)});
            \addplot[domain=-2:2,samples=20]({x},{.528});
            \addplot[domain=-2:2,samples=20]({x},{-.53});
            \addplot[domain=-sqrt(2):sqrt(2),samples=50,dashed,]({x},{0});
            \addplot[domain=0:1,samples=50,red]({x},{0});
            \node[below=.5pt of {(1,0)},outer sep=1pt]{$\boldsymbol{x}_1$};
            \addplot[domain=1:1.31175-.8,samples=50,dashed,blue]({x},{-.8096*(x-1.31175+.8)+0.35});
            \node[left =.5pt of {(1.31175-.8,0.405)},outer sep=1pt]{$\boldsymbol{x}'_2$};
            \addplot[domain=0:1.31175-.8,samples=50,dashed,blue]({x},{2/3*x});
            \node[below =.5pt of {(0,0)},outer sep=1pt]{$\boldsymbol{x}_0$};
            \addplot[domain=0:-.55,samples=50]({-sqrt(2)},{x});
            \node[above right =.5pt of {(-1.34,-.50)},outer sep=1pt]{$r_{2}$};
            \node[above =.5pt of {(-1.3,-1.8)},outer sep=1pt]{$\boldsymbol{x}_1=\begin{bmatrix}r_1\\0\\0\end{bmatrix}$};
            \node[above =.5pt of {(-.1,-1.8)},outer sep=1pt]{$\boldsymbol{x}'_2=\begin{bmatrix}x_{12}\\x_{22}\\x_{32}\end{bmatrix}$};
            \fill (1,0) circle[radius=2pt];
            \fill (0,0) circle[radius=2pt];
            \fill (1.31175-.8,0.35) circle[radius=2pt];
            \draw[{Latex[length=1.5mm]}-{Latex[length=1.5mm]}] (-1.41421356237,0)--(-1.41421356237,-.53);
    \end{axis}
\end{tikzpicture}


\begin{tikzpicture}
    \begin{axis}[
            xmin=-2,xmax=2,
            ymin=-2,ymax=1.2,
            ticks=none,width = 3in, height = 2.5in]
            \addplot[domain=90:3*90,samples=50]({.35*.4*cos(x)+sqrt(2)-.8},{(250/189)*.4*sin(x)});
            \addplot[domain=3*90:5*90,samples=50,dashed]({.35*.4*cos(x)+sqrt(2)-.8},{(250/189)*.4*sin(x)});
            \addplot[domain=90:3*90,samples=50]({.35*.4*cos(x)+sqrt(2)},{(250/189)*.4*sin(x)});
            \addplot[domain=3*90:5*90,samples=50,dashed]({.35*.4*cos(x)+sqrt(2)},{(250/189)*.4*sin(x)});
            \addplot[domain=0:.75,samples=50,blue]({x},{.2518*x});
           \addplot[domain=1:.75,samples=5,blue]({x},{-.65*(x-.75)+0.175});
            \addplot[domain=90:3*90,samples=50]({.35*.4*cos(x)-sqrt(2)},{(250/189)*.4*sin(x)});
            \addplot[domain=3*90:5*90,samples=50,dashed]({.35*.4*cos(x)-sqrt(2)},{(250/189)*.4*sin(x)});
            \addplot[domain=-2:2,samples=20]({x},{.528});
            \addplot[domain=-2:2,samples=20]({x},{-.53});
            \addplot[domain=-sqrt(2):sqrt(2),samples=50,dashed,]({x},{0});
            \addplot[domain=0:1),samples=50,red]({x},{0});
            \node[below=.5pt of {(1,0)},outer sep=1pt]{$\boldsymbol{x}_1$};
        
            \node[right =.5pt of {(.75,0.205)},outer sep=1pt]{$\boldsymbol{x}_2$};
            
            \node[below =.5pt of {(0,0)},outer sep=1pt]{$\boldsymbol{x}_0$};
            \addplot[domain=0:-.55,samples=50]({-sqrt(2)},{x});
            \node[above right =.5pt of {(-1.34,-.50)},outer sep=1pt]{$r_{2}$};
            \node[above =.5pt of {(-1.3,-1.8)},outer sep=1pt]{$\boldsymbol{x}_1=\begin{bmatrix}r_1\\0\\0\end{bmatrix}$};
            \node[above =.5pt of {(-.1,-1.8)},outer sep=1pt]{$\boldsymbol{x}_2=\begin{bmatrix}x_{12}\\r_2\\0\end{bmatrix}$};
        
            \fill (0,0) circle[radius=2pt];
            \fill (1,0) circle[radius=2pt];
            \fill (.75,0.175) circle[radius=2pt];
            \draw[{Latex[length=1.5mm]}-{Latex[length=1.5mm]}] (-1.41421356237,0)--(-1.41421356237,-.53);
    \end{axis}
\end{tikzpicture}

}

\vspace{.8mm}


\noindent\resizebox{\textwidth}{!}{

\begin{tikzpicture}
    \begin{axis}[
            xmin=-1.3,xmax=2.7,
            ymin=-2.7,ymax=1.7,
            ticks=none,width = 3in, height = 3.3in]
            \node[left =.5pt of {(0,0)},outer sep=.5pt]{$\boldsymbol{x}_0$};
            \addplot[domain=0:.75,samples=50,blue]({x},{.2518*x});
           \addplot[domain=1:.75,samples=5,blue]({x},{-.65*(x-.75)+0.175});
            \node[right=.5pt of {(1,0)},outer sep=1pt]{$\boldsymbol{x}_1$};
            \addplot[domain=.5:2.7,samples=50]({x},{.36*(x-1.9)});
            \addplot[domain=0:1,samples=50,teal,dashed]({x},{-.66*(x-1)-.66505});
            \addplot[domain=0:-.6505,samples=50,teal,dashed]({1},{x});            \addplot[domain=1:.75,samples=50,teal,dashed]({x},{-3.27*(x-1)-.6505});
            \addplot[domain=.72:1.37,samples=50,dashed]({x},{.6-.04});
            \addplot[domain=1.37:2.7,samples=50]({x},{.6-.04});
            \addplot[domain=-1.3:.5,samples=50]({x},{-.5});
            \addplot[domain=-.2:.72,samples=50,dashed]({x},{.36*(x+.9)});
            \addplot[domain=-1.3:-.2,samples=50]({x},{.36*(x+.9)});
            \addplot[domain=.5:2.7,samples=50]({x},{.36*(x-1.9)+.55+.21});
            \addplot[domain=.72:2.7,samples=50]({x},{1.13+.2});
            \addplot[domain=-1.3:.5,samples=50]({x},{-.5+.56+.2});
           \addplot[domain=-1.3:.72,samples=50]({x},{.36*(x+.9)+.55+.2});
           \addplot[domain=.5:2.7,samples=50]({x},{.36*(x-1.9)-.55-.2});
            \addplot[domain=.72:1.42,samples=50,dashed]({x},{.2-.16-.2});
            \addplot[domain=1.42:2.7,samples=50]({x},{.2-.16-.2});
            \addplot[domain=-1.3:.5,samples=50]({x},{-.5-.55-.2});
            \addplot[domain=-.2:.72,samples=50,dashed]({x},{.36*(x+.9)-.55-.2});
            \addplot[domain=-1.3:-.2,samples=50]({x},{.36*(x+.9)-.55-.2});
            \addplot[domain=-.4:.15,samples=50]({2.3},{x});
            \node[above =.5pt of {(2.5,-.27},outer sep=1pt]{$r_3$};
            \addplot[domain=.15:.685,samples=50]({2.3},{x});
            \node[above =.5pt of {(2.5,.27},outer sep=1pt]{$r_3$};
            \node[right =.5pt of {(.75,0.1705)},outer sep=1pt]{$\boldsymbol{x}_2$};
            \addplot[domain=0:1,samples=50,red]({x},{0});
            \node[above =.5pt of {(-1.3+.7,-1.8-.3-.4)},outer sep=1pt]{$\boldsymbol{x}_1=\begin{bmatrix}r_1\\0\\0\end{bmatrix}$};
            \node[above =.5pt of {(-.1+.7,-1.8-.3-.4)},outer sep=1pt]{$\boldsymbol{x}_2=\begin{bmatrix}x_{12}\\r_2\\0\end{bmatrix}$};
            \node[above =.5pt of {(1.1+.7,-1.8-.3-.4)},outer sep=1pt]{$\boldsymbol{x}'_3=\begin{bmatrix}x_{13}\\x_{23}\\x_{33}\end{bmatrix}$};
            \node[right =.5pt of {(1,-.6505)},outer sep=1pt]{$\boldsymbol{x}'_3$};
        
            \fill (1,-0.65) circle[radius=2pt];
            \fill (.75,0.175) circle[radius=2pt];
            \fill (1,0) circle[radius=2pt];
            \fill (0,0) circle[radius=2pt];
            \draw[{Latex[length=1.5mm]}-{Latex[length=1.5mm]}] (2.3,-.6)--(2.3,.15);
\draw[{Latex[length=1.5mm]}-{Latex[length=1.5mm]}] (2.3,.15)--(2.3,.89);
         \end{axis}
\end{tikzpicture}


\begin{tikzpicture}
    \begin{axis}[
            xmin=-1.3,xmax=2.7,
            ymin=-2.7,ymax=1.7,
            ticks=none,width = 3in, height = 3.3in]
            \node[left =.5pt of {(0,0)},outer sep=1pt]{$\boldsymbol{x}_0$};
           \addplot[domain=0:.75,samples=50,blue]({x},{.2518*x});
           \addplot[domain=1:.75,samples=5,blue]({x},{-.65*(x-.75)+0.175});
            \node[right=.5pt of {(1,0)},outer sep=1pt]{$\boldsymbol{x}_1$};
            \addplot[domain=0:.7505,samples=50,teal]({1},{x});
            \addplot[domain=0:1,samples=50,teal]({x},{.78*(x-1)+.7505});
            \addplot[domain=1:.75,samples=50,teal]({x},{2.4*(x-1)+.7505});
            \addplot[domain=.5:2.7,samples=50]({x},{.36*(x-1.9)});
            \addplot[domain=.72:1.37,samples=50,dashed]({x},{.6-.04});
            \addplot[domain=1.37:2.7,samples=50]({x},{.6-.04});
            \addplot[domain=-1.3:.5,samples=50]({x},{-.5});
            \addplot[domain=-.2:.72,samples=50,dashed]({x},{.36*(x+.9)});
            \addplot[domain=-1.3:-.2,samples=50]({x},{.36*(x+.9)});
            \addplot[domain=.5:2.7,samples=50]({x},{.36*(x-1.9)+.55+.21});
            \addplot[domain=.72:2.7,samples=50]({x},{1.13+.2});
            \addplot[domain=-1.3:.5,samples=50]({x},{-.5+.56+.2});
           \addplot[domain=-1.3:.72,samples=50]({x},{.36*(x+.9)+.55+.2});
           \addplot[domain=.5:2.7,samples=50]({x},{.36*(x-1.9)-.55-.2});
            \addplot[domain=.72:1.42,samples=50,dashed]({x},{.2-.16-.2});
            \addplot[domain=1.42:2.7,samples=50]({x},{.2-.16-.2});
            \addplot[domain=-1.3:.5,samples=50]({x},{-.5-.55-.2});
            \addplot[domain=-.2:.72,samples=50,dashed]({x},{.36*(x+.9)-.55-.2});
            \addplot[domain=-1.3:-.2,samples=50]({x},{.36*(x+.9)-.55-.2});
            \addplot[domain=.15:.685,samples=50]({2.3},{x});
            \node[above =.5pt of {(2.5,.27},outer sep=1pt]{$r_3$};
            \addplot[domain=-.4:.15,samples=50]({2.3},{x});
            \node[above =.5pt of {(2.5,-.27},outer sep=1pt]{$r_3$};
            \node[right =.5pt of {(.742,0.205)},outer sep=1pt]{$\boldsymbol{x}_2$};
            \addplot[domain=0:1,samples=50,red]({x},{0});
            \node[above =.5pt of {(-1.3+.7,-1.8-.3-.4)},outer sep=1pt]{$\boldsymbol{x}_1=\begin{bmatrix}r_1\\0\\0\end{bmatrix}$};
            \node[above =.5pt of {(-.1+.7,-1.8-.3-.4)},outer sep=1pt]{$\boldsymbol{x}_2=\begin{bmatrix}x_{12}\\r_2\\0\end{bmatrix}$};
            \node[above =.5pt of {(1.1+.7,-1.8-.3-.4)},outer sep=1pt]{$\boldsymbol{x}_3=\begin{bmatrix}x_{13}\\x_{23}\\r_3\end{bmatrix}$};
            \node[right =.5pt of {(1,.7505)},outer sep=1pt]{$\boldsymbol{x}_3$};
            \fill (.75,0.175) circle[radius=2pt];
            \fill (1,0.75) circle[radius=2pt];
            \fill (1,0) circle[radius=2pt];
            \fill (0,0) circle[radius=2pt];
\draw[{Latex[length=1.5mm]}-{Latex[length=1.5mm]}] (2.3,-.6)--(2.3,.15);
\draw[{Latex[length=1.5mm]}-{Latex[length=1.5mm]}] (2.3,.15)--(2.3,.89);
         \end{axis}
\end{tikzpicture}

}

\caption{Constructing the $3$-simplex $\mathcal{S}$ in $\mathbb{R}^3$, with the Brownian path omitted for clarity. The three rows of figures correspond to the three stages. The left and right columns correspond, respectively, to before and after reorientation.}

\label{fig:procedure}

\end{figure}
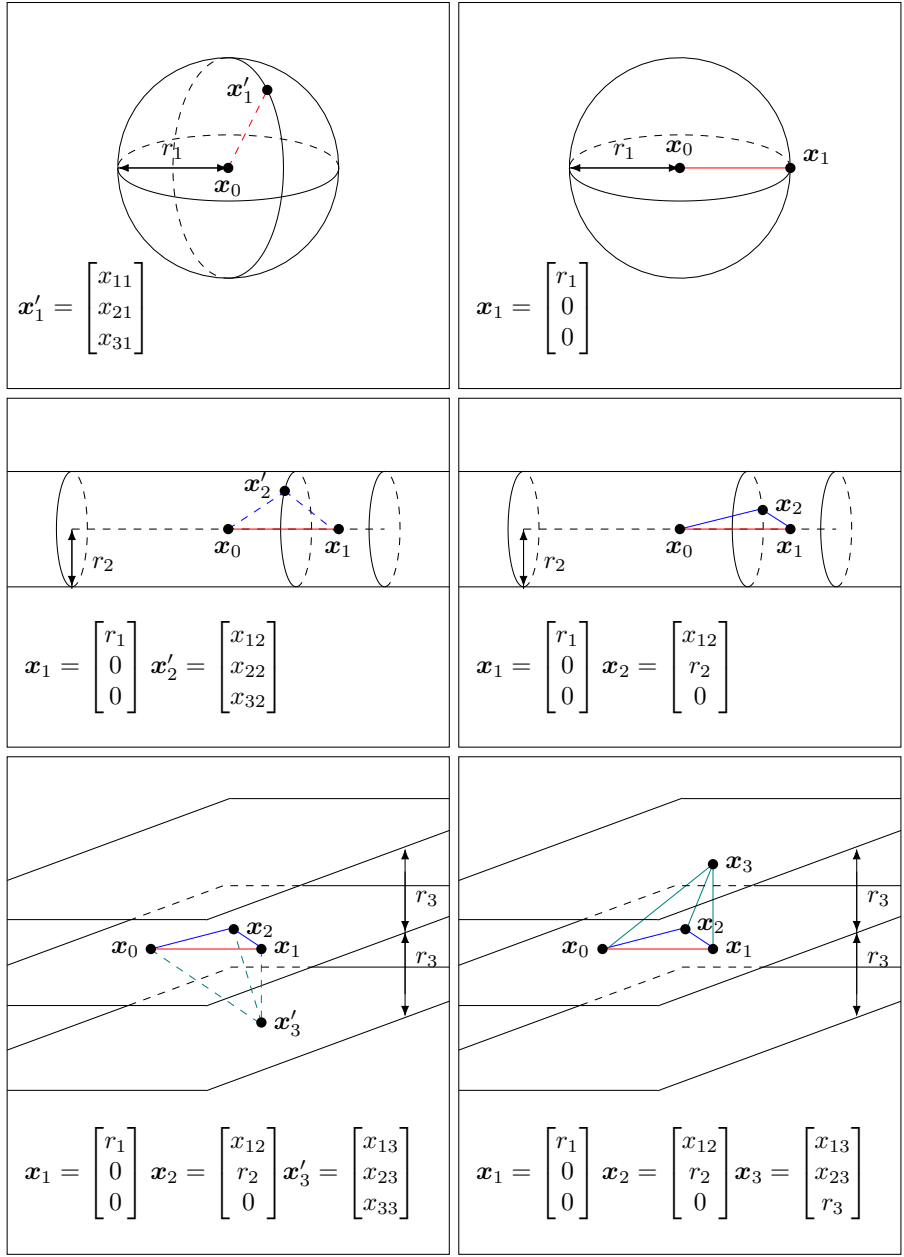


To compute the right-hand side of \eqref{eq:simplex_bound}, we note that by construction,
\[
\proj(\boldsymbol{W}_{T_{j-1}},\mathcal{X}_{j-1})=\boldsymbol{0},~1\leq j\leq n.
\]
Moreover, since $\proj(\boldsymbol{W},\mathcal{X}_{j-1})$ is Brownian motion in $n-j+1$ dimensions, it follows for all $1\leq j\leq n$ that $T_j-T_{j-1}$ is the first exit time from an $n-j+1$ dimensional open ball of radius $r_j$ by $n-j+1$ dimensional Brownian motion starting at the center of the ball. We can now use \eqref{eq:ball_moment} to write 
\begin{equation}\label{eq:simplex_time}
\mathbb{E}[T_n]=\mathbb{E}\left[\sum_{j=1}^n (T_j-T_{j-1})\right]=\sum_{j=1}^n \frac{r_j^2}{n-j+1}.
\end{equation}

It follows from part (i) of Proposition \ref{prop:distributional_equalities} that $\mathbb{E}[\Theta_y^V]$ is an increasing function of the volume $y$. Hence, \eqref{eq:simplex_bound} and \eqref{eq:simplex_volume} together imply that we have
\[
\mathbb{E}\left[\Theta_1^V\right]\leq \mathbb{E}\left[\Theta_{V(\mathcal{S})}^V\right]\leq \mathbb{E}[T_n]
\]
whenever the radii satisfy $r_1\cdots r_n\geq n!$. Now we can conclude from \eqref{eq:simplex_time} and Lemma \ref{lem:optimization} that
\begin{align*}
\mathbb{E}\left[\Theta_1^V\right]&\leq \inf\left\{\sum_{j=1}^n \frac{r_j^2}{n-j+1}\middle|(r_1,\dots,r_n)\in \mathbb{R}_{> 0}^n\text{ and }r_1\cdots r_n\geq n!\right\}\\
&=n\sqrt[n]{n!}.
\end{align*}
\end{proof}

\begin{proof}[Proof of Theorem \ref{thm:inverse_area}]
The lower bound is a straightforward consequence of part (ii) of Proposition \ref{prop:distributional_equalities} together with Jensen's inequality and Eldan's formula \eqref{eq:Eldan} for the expected value of the surface area. In particular, these results allow us to write
\begin{align*}
\mathbb{E}\left[\Theta_1^S\right]=\mathbb{E}\left[S_1^{-2/(n-1)}\right] &\geq \left(\frac{2(2\pi)^{(n-1)/2}}{\Gamma(n)}\right)^{-2/(n-1)}\\
&=\frac{1}{2\pi}\left(\frac{1}{2}\Gamma(n)\right)^{2/(n-1)}.
\end{align*}

To prove the upper bound, first note that at time $\Theta_1^V$, the surface area of the convex hull must be at least that of a ball in $\mathbb{R}^n$ with volume $1$. From the scaling property of Lebesgue measure, we can deduce that a ball in $\mathbb{R}^n$ with volume $1$ has surface area $\omega_n\kappa_n^{(1-n)/n}$. Thus, part (ii) of Proposition \ref{prop:distributional_equalities} implies
\begin{equation}\label{eq:VS_comparison}
\mathbb{E}\left[\Theta_1^V\right]\geq\left( \omega_n\kappa_n^{(1-n)/n}\right)^{-2/(n-1)}\mathbb{E}\left[\Theta_1^S\right].
\end{equation}
Now combining \eqref{eq:VS_comparison} with \eqref{eq:unit_ball} and the upper bound from Theorem \ref{thm:inverse_volume} leads to
\begin{align*}
\mathbb{E}\left[\Theta_1^S\right]&\leq \left( \frac{2\pi^{n/2}}{\Gamma(n/2)}\left(\frac{\pi^{n/2}}{\Gamma(1+n/2)}\right)^{(1-n)/n}\right)^{2/(n-1)} n\sqrt[n]{n!}\\
&=\left(\frac{\sqrt{\pi}\, n}{\sqrt[n]{\Gamma(1+n/2)}}\right)^{2/(n-1)} n\sqrt[n]{n!}.
\end{align*}
\end{proof}

\subsection{Diameter and its inverse process}\label{sec:diameter}

\begin{proof}[Proof of Theorem \ref{thm:diameter}]
We start by proving the lower bound. Note that
\begin{align}
D_1 = \sup_{0\leq s\leq t\leq 1}\|\boldsymbol{W}_s-\boldsymbol{W}_t\|& \geq \sup_{0\leq t\leq 1} \|\boldsymbol{W}_t\|\nonumber \\
&\stackrel{d}{=}\frac{1}{\sqrt{\tau_1}},\label{eq:diameter_law}
\end{align}
where we used $\tau_1$ to denote the exit time of the unit ball in $\mathbb{R}^n$ by $n$-dimensional Brownian motion starting from the center. The equality in distribution in \eqref{eq:diameter_law} is similar to those from Proposition \ref{prop:distributional_equalities} and follows from Brownian scaling; see also \cite[Equation (11)]{max_Bessel}. Taking the expected value of \eqref{eq:diameter_law} while applying Jensen's inequality on the right-hand side and using \eqref{eq:ball_moment} leads to
\[
\mathbb{E}[D_1]\geq \frac{1}{\sqrt{\mathbb{E}[\tau_1]}}= \sqrt{n}.
\]

To prove the upper bound, we consider the hyperrectangle $\mathcal{R}$ circumscribed around $\mathcal{H}_1$ that has each edge parallel to a coordinate axis and use the diagonal of $\mathcal{R}$ to bound $D_1$ from above. This idea was first used by McRedmond and Xu in the planar case; see \cite[Proposition 5]{McRedmond_Xu}. More precisely, we have
\begin{equation*}
D_1 \leq \sqrt{\sum_{i=1}^n\left(\sup_{0\leq t\leq 1} W_t^{(i)}-\inf_{0\leq t\leq 1} W_t^{(i)}\right)^2}.
\end{equation*}
The quantity $\sup_{0\leq t\leq 1} W_t^{(i)}-\inf_{0\leq t\leq 1} W_t^{(i)}$ is nothing but the range of the $i$\textsuperscript{th} coordinate of $\boldsymbol{W}$. In particular, its second moment was computed by Feller in \cite{Feller} and was shown to be $4\log 2$. Hence, it follows from Jensen's inequality that
\begin{align}
\mathbb{E}[D_1] &\leq \sqrt{n\, \mathbb{E}\left[\left(\sup_{0\leq t\leq 1} W_t^{(1)}-\inf_{0\leq t\leq 1} W_t^{(1)}\right)^2\right]}\nonumber \\
&= \sqrt{4n\log 2}. \label{eq:Feller}
\end{align}
\end{proof}

\begin{proof}[Proof of Theorem \ref{thm:inverse_diameter}]
For the lower bound, we can use part (iii) of Proposition \ref{prop:distributional_equalities} along with Jensen's inequality and the upper bound from Theorem \ref{thm:diameter} to write
\[
\mathbb{E}\left[\Theta_1^D\right] \geq \frac{1}{\mathbb{E}[D_1]^2} \geq \frac{1}{4n\log 2}.
\]

For the upper bound, note that as soon as $\boldsymbol{W}$ exits a ball of radius $1$, its convex hull contains a line segment of length at least $1$. Hence, the diameter of the convex hull at this time is at least $1$. Thus, $\Theta_1^D \leq \tau_1$, and from \eqref{eq:ball_moment} we get
\[
\mathbb{E}\left[\Theta_1^D\right]\leq \mathbb{E}[\tau_1]=\frac{1}{n}.
\]
\end{proof}


\subsection{Circumradius and its inverse process}\label{sec:circumradius}

Before proving Theorem \ref{thm:circumradius}, we need a lemma that equates the diameter and twice the circumradius of a centrally symmetric compact set in $\mathbb{R}^n$. A set $\mathcal{C}\subset\mathbb{R}^n$, not necessarily convex, is \emph{centrally symmetric} with respect to the point $\boldsymbol{p}\in\mathbb{R}^n$ if 
\begin{equation}\label{eq:central_symmetry}
\{2\boldsymbol{p}-\boldsymbol{x}:\boldsymbol{x}\in\mathcal{C}\}=\mathcal{C}.
\end{equation}
The fact that the following lemma isn't true for general compact $\mathcal{C}$ is part of the substance of Jung's theorem; see \cite[Theorem 3.3]{Gruber}.

\begin{lemma}\label{lem:diameter_circumradius}
Let $\mathcal{C}\subset\mathbb{R}^n$ be a centrally symmetric compact set. Then 
\begin{equation}\label{eq:diameter_circumradius}
D(\mathcal{C})=2R(\mathcal{C}).
\end{equation}
In particular, \eqref{eq:diameter_circumradius} holds for any hyperrectangle.
\end{lemma}

\begin{proof}[Proof of Lemma \ref{lem:diameter_circumradius}]
Without loss of generality, we can assume that $\mathcal{C}$ is centrally symmetric with respect to the origin. By compactness, there exists points $\boldsymbol{x},\boldsymbol{y}\in\mathcal{C}$ with $\|\boldsymbol{x}-\boldsymbol{y}\|=D(\mathcal{C})$. Hence, $D(\mathcal{C})\leq 2R(\mathcal{C})$, for otherwise, $\boldsymbol{x}$ and $\boldsymbol{y}$ could not both fit inside the circumball. To prove that the other inequality holds, define $\rho:=\sup\{\|\boldsymbol{x}\|:\boldsymbol{x}\in\mathcal{C}\}$. Certainly $R(\mathcal{C})\leq \rho$, since $\mathcal{C}$ is contained in the closed ball of radius $\rho$ centered at the origin. By compactness, we know there exists some $\boldsymbol{x}\in\mathcal{C}$ with $\|\boldsymbol{x}\|=\rho$. By central symmetry, $-\boldsymbol{x}\in\mathcal{C}$, so we have $D(\mathcal{C})\geq 2\rho$. Putting these two inequalities together gives $D(\mathcal{C})\geq 2R(\mathcal{C})$, which proves the first claim.

We prove the last claim by establishing the central symmetry of any hyperrectangle $\mathcal{R}\subset\mathbb{R}^n$. Without loss of generality, we can assume that each edge of $\mathcal{R}$ is parallel to a coordinate axis and that $\mathcal{R}$ is centered at the origin. More precisely, 
\begin{equation}\label{eq:hyperrectangle}
\mathcal{R}=\big\{(x_1,\dots,x_n)\in\mathbb{R}^n:|x_1|\leq w_1,\dots, |x_n|\leq w_n\big\},
\end{equation}
where $w_1,\dots,w_n\geq 0$ are the coordinate half-widths of $\mathcal{R}$. It is clear from \eqref{eq:hyperrectangle} that $\boldsymbol{x}\in\mathcal{R}$ if and only if $-\boldsymbol{x}\in\mathcal{R}$. Hence, $\mathcal{R}$ satisfies \eqref{eq:central_symmetry} with $\boldsymbol{p}=\boldsymbol{0}$.
\end{proof}

\begin{proof}[Proof of Theorem \ref{thm:circumradius}]
In light of Theorem \ref{thm:diameter}, the lower bound follows immediately from the inequality $D(\mathcal{C})\leq 2R(\mathcal{C})$ that holds for any compact set $\mathcal{C}\subset\mathbb{R}^n$; refer to the proof of Lemma \ref{lem:diameter_circumradius} for an explanation of this inequality. 

For the upper bound, we consider the hyperrectangle $\mathcal{R}$ circumscribing $\mathcal{H}_1$ that has each edge parallel to a coordinate axis. Similarly to the proof of Theorem \ref{thm:diameter}, we can write 
\begin{align}
R_1\leq R(\mathcal{R})&=\frac{1}{2}D(\mathcal{R})\nonumber \\
&=\frac{1}{2}\sqrt{\sum_{i=1}^n\left(\sup_{0\leq t\leq 1} W_t^{(i)}-\inf_{0\leq t\leq 1} W_t^{(i)}\right)^2}, \label{eq:diagonal}
\end{align}
where the first equality follows from Lemma \ref{lem:diameter_circumradius}. Taking  the expected value of \eqref{eq:diagonal}, while using Jensen's inequality on the right-hand side along with Feller's second moment calculation from \eqref{eq:Feller}, results in
\[
\mathbb{E}[R_1]\leq \frac{1}{2}\sqrt{4n\log 2}.
\]
\end{proof}

\begin{proof}[Proof of Theorem \ref{thm:inverse_circumradius}]
For the lower bound, we can use part (iv) of Proposition \ref{prop:distributional_equalities} along with Jensen's inequality and the upper bound from Theorem \ref{thm:circumradius} to write
\[
\mathbb{E}\left[\Theta_1^R\right] \geq \frac{1}{\mathbb{E}[R_1]^2} \geq \frac{1}{n\log 2}.
\]

For the upper bound, note that the trivial inequality $D(\mathcal{H}_t)\leq 2R(\mathcal{H}_t)$ implies that the circumradius of the convex hull will attain $1$ no later than when its diameter attains $2$. Hence, we can use Proposition \ref{prop:distributional_equalities} and Theorem \ref{thm:inverse_diameter} to write
\begin{equation*}
\mathbb{E}\left[\Theta_1^R\right]\leq \mathbb{E}\left[\Theta_2^D\right]\leq 4\frac{1}{n}.
\end{equation*}
\end{proof}

\bibliography{hull_bib}
\bibliographystyle{amsplainabbrev}

\end{document}